\documentclass[12pt]{amsart}
\usepackage{mathtools}
\usepackage{amsfonts}
\usepackage{mathrsfs}
\usepackage{bigints}
\usepackage{enumitem}
\usepackage[latin1]{inputenc}
\usepackage{amsmath,amssymb}
\usepackage{latexsym}
\usepackage{epstopdf}
\usepackage{geometry}   
\usepackage{xcolor}
\usepackage[active]{srcltx}
\usepackage{graphicx}
\usepackage{epstopdf}
\usepackage[
bookmarks=true,         
bookmarksnumbered=true, 
colorlinks=true, pdfstartview=FitV, linkcolor=blue, citecolor=blue,
urlcolor=blue]{hyperref}

\newtheorem {theorem}{Theorem}[section]
\newtheorem {proposition}{Proposition}[section]

\newtheorem{lemma}{Lemma}[section]

\numberwithin{equation}{section}

\def\R{{\mathbb R}}

\begin{document}

\title[]{Gaussian fluctuations of spatial averages of a system of stochastic heat equations}

\author{David Nualart, Bhargobjyoti Saikia}
 
\thanks{David Nualart was supported by the NSF grant DMS-2054735.}

\address{David Nualart: Department of Mathematics, University of Kansas, 1450 Jayhawk Blvd., Lawrence, KS 66045. USA.}
\email{nualart@ku.edu}   

\address{Bhargobjyoti Saikia: Department of Mathematics, University of Kansas, 1450 Jayhawk Blvd., Lawrence, KS 66045. USA.}
\email{bhargobjs@gmail.com}

\date{\today}

\date{}
\maketitle
 \begin{abstract} 
 We consider a system of $d$ non-linear stochastic heat equations    driven by an $m$-dimensional space-time white noise on $\R_+\times \R$. In this paper we study the asymptotic behavior of spatial averages over large intervals $[-R,R]$.  We establish a rate of convergence to a multivariate normal distribution in the Wasserstein distance and a functional central limit theorem.
  \end{abstract}
  \section{Introduction}\label{section settings}
In this paper we are interested in the following system of stochastic heat equations 

\begin{align}\label{equation}
\frac{\partial u_i}{\partial t}(t,x)= \frac {\partial ^2 u_i} {\partial x^2}(t,x)+\sum_{j=1}^m \sigma_{ij}(u(t,x))\dot{W_j}(t,x) ,
\end{align}
where $u=(u_1,\dots,u_d)$  is a $d$-dimensional random vector depending on the space and time parameters $(t,x) \in \mathbb{R}_+\times \mathbb{R}$.
We assume a constant  initial condition $u_i(0,x)=1$, $1\le i\le d$. For all $1\le i \le d$ and $1\le j\le m$,  $\sigma_{ij}:\mathbb{R}^d\rightarrow \mathbb{R}$ are Lipschitz functions. The noise 
$\dot{W}:=(\dot{W}_1,\dots, \dot{W}_m)$ is a vector of $m$ independent  Gaussian space-time  white noises on $\R_+\times\R$. Note that $d\leq m$. 

It is well-known (see, for instance, \cite{W}) that  equation \eqref{equation} has a unique mild solution $u$ such that for each $i=1,\dots, d$, it satisfies the equation
\begin{equation}
\label{solution}u_i(t,x)=1+\sum_{j=1}^m\int_0^t \int_{\mathbb{R}}p_{t-s}(x-y)\sigma_{ij}(u(s,y))W_j(ds,dy),
\end{equation}
where in the right-hand side, the stochastic integral is interpreted  in the It\^o-Walsh sense, and 
$p_t(x)=(2\pi t)^{-\frac12} e^{-x^2/2t}$, $(t,x) \in \mathbb{R}_+\times \mathbb{R}$, is the heat kernel. This type of systems of stochastic heat equations have been extensively studied and, for instance, we refer to \cite{DKN} for a detailed analysis of hitting probabilities.

Fix a time $t>0$. The goal of this paper is to study the asymptotic behaviour, as  $R$ tends to infinity, of the the spatial averages $F^R(t):=(F^R_1(t),\dots,F_d^R(t))$, 
where,  for $i=1,\dots, d$,
\begin{align}
\label{avg}F_i^{R}(t):= \frac{1}{\sqrt{R}}\left(\int_{-R}^R u_i(t,x)dx-2R\right).
\end{align}
In the case $d=m=1$ functional and quantitative central limit theorems have been established in \cite{CLSHE} using the techniques of Malliavin calculus combined with  Stein's method for normal approximations.   Our work is an extension of that paper to a system of equations.   In this case, the limit in law of  the family of $d$-dimensional random  processes $\{F^R(t), t\in [0,T]\}$ is a $d$-dimensional Gaussian martingale. For the rate of convergence, instead of the total variation distance considered in \cite{CLSHE}, we will deal with the Wasserstein distance, because it is more appropriate in the multidimensional case,  and we will establish a rate of convergence of the order $R^{-1/2}$, assuming a suitable non-degeneracy condition.

The paper is organized as follows. Section 2 contains the main results of this paper. In Section 3 we introduce some preliminaries and in Section 4  we present some important results, which we use later in the proof of the main theorems. Finally, the proof of the main results is given in Section 5.

\section{Main results}
The spatial integral $\int_{-R}^R u_i(t,x)dx$ behaves like a sum of i.i.d random variables and it will be natural to expect a central limit theorem to hold in this situation. Along the paper, we will denote by $C(t)$ the symmetric and nonnegative definite  $d\times d$ matrix given by
\begin{equation}\label{LimitMatrix}
    C_{ij}(t):=2\sum_{k=1}^m\int_0^{t}\eta^{(k)} _{ij}(r)dr, \quad 1\le i,j \le d,
\end{equation}
where 
\begin{equation}
\label{nu}\eta^{(k)}_{ij}(r):=E\left[ \sigma_{ik}(u(r,x)) \sigma_{jk}(u(r,x))\right]. 
 \end{equation}
 Notice that $\eta^{(k)}_{ij}(r)$  does not depend on $x$ because for any fixed $t>0$,
 the process $\{ u(t,x), x\in \mathbb{R}\}$ is stationary   (see, for instance,   \cite[Lemma 7.1]{CKNP}). We will show    that  for each $1\le i,j \le d$
\begin{equation}
\label{convC}\lim_{R\rightarrow \infty} E\left(F_i^{R}(t)F_j^{R}(t)\right)=C_{ij}(t).
\end{equation}

We will make use of the following non-degeneracy condition:

\noindent
{\bf (H1)} The vector space spanned by the family of vectors $\{ (\sigma_{1k} (\overline{1}), \dots, \sigma_{dk} (\overline{1})), k=1,\dots,m\}$
 has dimension $d$, where $\overline {1}=(1,\dots, 1)$.

\smallskip
 Notice that in hypothesis {\bf (H1)} the vector $\overline {1}=(1,\dots, 1)$ is the initial condition, which is supposed to have constant components to ensure the spatial stationarity of the solution. 
 
 \smallskip
Our first result is the following quantitative central limit theorem.

\begin{theorem}\label{MT}
Suppose that $u(t,x)=(u_1(t,x), \dots, u_d(t,x))$ is the mild solution to equation \eqref{equation} and let $F^R(t)$ be given as in \eqref{avg}. 
Assume condition {\bf (H1)}. Let $d_W$ denote the Wasserstein distance defined in  \eqref{wasser} below and let $N(t)$ be a $d$-dimensional random vector with the multivariate normal distribution 
 $ \mathcal{N}(0,C(t))$ with mean zero and covariance matrix $C(t)$  defined  by \eqref{LimitMatrix}.
  Then, there is a constant $c>0$ depending on $t$, such that for any $R\ge 1$,
   \begin{align*}
    d_W(F^R(t),N(t)) &\leq \frac{c}{\sqrt{R}}.
\end{align*}
\end{theorem}

Our second result is a functional central limit theorem.
\begin{theorem}\label{CLT} Let 
$u(t,x)=(u_1(t,x),\dots, u_d(t,x))$ be the mild solution to equation \eqref{equation} and let  $\eta^{(k)}_{ij}(r)$  be as in \eqref{nu}. Then  for each  $T>0$
\begin{align}\label{limit}
\left(\frac{1}{\sqrt{R}}\int_{-R}^{R}u(t,x)dx-2R\right)_{t\in[0,T]}&\rightarrow \left(M(t)\right)_{t\in [0,T]}, 
\end{align}
as R tends to infinity, where $(M(t))_{t\in [0,T]}$ is a $d$-dimensional centered Gaussian martingale which satisfies
 \begin{align*}
    E(M_i(t)M_j(s))=\sum_{k=1}^m\int_0^{t\wedge s}2\eta^{(k)}_{ij}(r)dr,
\end{align*}
for each $s,t \in [0,T]$ and all $1\le i,j \le d$, and the convergence holds in distribution.
\end{theorem}

\section{Preliminaries}
Let us first introduce the white noise on $\mathbb{R}_{+}\times \mathbb{R}$.
We denote by $\mathcal{B}_b(\mathbb{R}_{+}\times \mathbb{R})$ the collection of Borel sets $ A\subset \mathbb{R}_+\times \mathbb{R}$ with finite Lebesgue measure, denoted by $|A|$. Consider a centered Gaussian family of random variables $W=\{W_j(A),A\in\mathcal{B}_b, 1\le j \le  m\}$, defined in a complete probability space $(\Omega,\mathcal{F},P)$, with covariance 
\[
E[W_j(A)W_k(B)]=\mathbf{1}_{\{j=k\}}|A\cap B|.
\]
We assume that the $\sigma$-algebra $\mathcal{F}$ is generated by $W$ and the $P$-null sets. For any $t\ge 0$, we denote by $\mathcal{F}_t$ the $\sigma$-algebra generated by the random variables
\[
\{W_j([0,s]\times A): 1\le j\le m, 0\le s\le t, A\in \mathcal{B}_b(\mathbb{R})\}.
\]

 As proved in \cite{W}, for any $m$-dimensional adapted random field $\{X(s,y),(s,y)\in \mathbb{R}_+\times \mathbb{R}\}$ that is jointly measurable and
\begin{align}
\label{l2}\int_0^{\infty}\int_{\mathbb{R}}E[|X(s,y)|^2]dyds<\infty,\end{align}
the following stochastic integral 
\[ 
\sum_{j=1}^m \int_0^{\infty}\int_{\mathbb{R}}X_j(s,y)W_j(ds,dy)
\] 
is well defined.
Next we will introduce the basic elements of Malliavin calculus which  are required to prove our results.

\subsection{Malliavin Calculus}
In this section we will discuss some basic facts about  the Malliavin calculus associated with $W$. We refer the reader to \cite{N} for a detailed account on the Malliavin calculus with respect to a Gaussian process.

Consider the Hilbert space $\mathcal{H}=L^2(\mathbb{R}_+\times \mathbb{R};\R^m)$. The Wiener integral 
\[
W(h)=\sum_{j=1}^m \int_0^{\infty}\int_{\mathbb{R}} h_j(t,x)W^j(dt,dx)
\] 
provides an isometry between the Hilbert space $\mathcal{H}$ and $L^2(\Omega)$. In this sense $\{W(h),h\in \mathcal{H} \}$ is an isonormal Gaussian process.

We denote by $C^{\infty}_{p}(\mathbb{R}^n)$ the space of infinitely differentiable functions with all their partial derivatives having at most polynomial growth at infinity. Let $\mathcal{S}$ be the space of simple  and smooth random variables of the form 
\[ 
F=f(W(h^{(1)}),\dots,W(h^{(n)}))
\]
for $f\in C^{\infty}_{p}(\mathbb{R}^n) $ and $h^{(i)} \in \mathcal{H}, 1\le i\le n$. Then, the Malliavin derivative $DF$ is the $\mathcal{H}$-valued random variable defined by 
\begin{align}D_{s,y}^k F &=\sum_{i=1}^n \frac{\partial f}{\partial x_i}(W(h^{(1)}),\dots,W(h^{(n)}))h_k^{(i)}(s,y),
\end{align}
where $1\le k \le m$ and $(s,y) \in \R_+\times \R$.
The derivative operator $D$ is a closable operator  with values in $L^p(\Omega;\mathcal{H})$ for any $p\ge 1$. For any $p\ge 1,$
let $\mathbb{D}^{1,p}$ be the completion of $\mathcal{S}$ with respect to the norm 
\[
\|F\|_{1,p}=(E|F|^p+E||DF||^p_{\mathcal{H}})^{\frac{1}{p}}.
\]
We denote by $\delta$ the adjoint of the derivative operator given by the duality formula
\begin{align}
\label{duality} E(\delta(u)F)&=E(\langle u,DF\rangle_{\mathcal{H}})
\end{align}
for any $F\in \mathbb{D}^{1,2}$, and any $u\in L^2(\Omega;\mathcal{H})$ in the domain of $\delta$, denoted by  $ {\rm Dom}\ \delta$. The operator $\delta$ is also called Skorohod integral and in the  Brownian motion case it coincides with an extension of the It\^o integral  introduced by Skorohod (see \cite{Sk}). More generally, in the context of the space-time white noise $W$, any $m$-dimensional adapted random field $X$ which is jointly measurable and satisfies \eqref{l2} belongs to the domain of $\delta$ and $\delta(X)$ coincides with the It\^o-Walsh integral
\[ 
\delta(X)=\sum_{j=1}^m \int_0^{\infty}\int_{\mathbb{R}}X_j(s,y)W_j(ds,dy).
\]
As a consequence, the mild equation to equation  \eqref{equation} can be written as
\begin{equation}\label{kernel}
  u_i(t,x)=1+ \delta(\mathbf{1}_{[0,t]}(\bullet)p_{t-\bullet}(x-*)\sigma_i(u(\bullet,*)),
\end{equation}
where $(\bullet, *)$ denotes a variable in $\R_+\times \R$ and $\sigma_i=(\sigma_{i1},\dots,\sigma_{im})$.
It is known that for any $(t,x)\in \mathbb{R}_+\times \mathbb{R}$ and each $i=1, \dots, d$, $u_i(t,x)$   belongs to $\mathbb{D}^{1,p}$ for any $p\ge2$ and the derivative satisfies the following linear equation for $t\ge s$ and $1\le \ell \le m$,
\begin{align}\label{deriv}
    D^\ell_{s,y}u_i(t,x)&= p_{t-s}(x-y)\sigma_{i\ell}(u(s,y))\nonumber\\& \qquad \qquad +\sum_{j=1}^m \sum_{k=1}^d\int_s^t\int_{\mathbb{R}}p_{t-r}(x-z)\Sigma_{ij}^{(k)}(r,z)D_{s,y}^\ell u_k(r,z)W_j(dr,dz),
    \end{align}
where $\Sigma_{ij}^{(k)}(r,z)$ is an adapted process, bounded by the Lipschitz constant of $\sigma_{ij}$. If $\sigma_{ij}$ is continuously differentiable, then $\Sigma_{ij}^{(k)}(r,z)=\frac{\partial\sigma_{ij}(u(r,z))}{\partial x_k}$.  This result is proved in Proposition 2.4.4 of \cite{N} in the case of Dirichlet boundary condition on [0,1] and with $d=m=1$.  The proof can be easily extended to the multidimensional case and with the space variable on $\mathbb{R}$. We also refer to \cite{CHN, NQ}, where this result is used when the coefficient is continuously differentiable.

We use Stein method, which is a probabilistic technique which allows us to measure the distance between a probability distribution and a normal distribution, to prove our first theorem. With this aim, we next define the Wasserstein distance.

\subsection{Wasserstein Distance}
Let $F$ and $G$ denote  two integrable  $d$-dimensional random  vectors  defined on the probability space $(\Omega,\mathcal{F},\mathbb{P})$. Then the Wasserstein distance is defined as 
\begin{align}\label{wasser}
d_{W}(F,N)&=\sup_{h\in {\rm Lip}(1)}|E[h(F)]-E[h(N)]|.
\end{align}
 Here, ${\rm Lip}(K)$ stands for the set of functions $h : \mathbb{R}^d\rightarrow \mathbb{R}$ that are Lipschitz with
constant $K > 0$, that is, satisfying $|h(x)-h(y)| \leq K|x-y|$ for all $x, y \in \mathbb{R}^d$.

\section{Basic results}
The next two results provide  upper bounds for the Wasserstein distance between a    $d$-dimensional random vector
whose components can be expressed as divergences and   a random vector  with a $d$-dimensional Gaussian distribution.

\begin{proposition}\label{bound1}
Let $F=(F_1,\dots, F_d)$, with $F_i:=\delta(v_i)$, where $v_i\in {\rm Dom}(\delta), i=1,\dots,d$. Suppose that $ N$ is a $d$-dimensional Gaussian centered vector with an invertible covariance matrix $A$. Then  
\begin{equation} \label{EQ1}
    d_W(F,N)\leq   \sqrt{d}\|A^{-1}\|_{\rm op} \| A\|_{\rm op}^{\frac12} \sqrt{   \sum_{i,j=1}^d
     E(|A_{ij}-\langle v_i, DF_j\rangle_{\mathcal{H}}|^2)},
    \end{equation}
    where $\| \cdot \|_{\rm op}$ denotes the operator norm of a matrix.
\end{proposition}

\begin{proof} Since $v_i \in {\rm Dom}(\delta)$, the random variables $F_i$ are square integrable. So, 
using  Theorem 4.4.1 in \cite[page 85]{NAMC}, we have 
 \[
 d_W(F,N) \leq \sup_{f\in \mathcal{F}^d_{W}(A)}|E\left[\langle A,{\rm Hess} f(F)\rangle_{HS}\right]-E\left[\langle F,\nabla f(F)\rangle_{\mathbb{R}^d}\right]|,
 \]
 where  $ ({\rm Hess} f)_{ij}=\frac{\partial^2}{\partial_i \partial_j}f$ and $\mathcal{F}^d_{W}(A)$ denotes the class of twice continuously differentiable functions  $f:\mathbb{R}^d\rightarrow  \mathbb{R}$ such that
    \[
    \sup_{x\in \mathbb{R}^d}   \|{\rm Hess} f(x)\|_{HS}\leq\sqrt{d}\|A^{-1}\|_{op} \| A\|_{op}^{\frac12}.
    \]
Consider the expression
\begin{align*}
  \Phi &:=  E\left[\langle A,{\rm Hess} f(F)\rangle_{HS}\right]-  E\left[\langle F,\nabla f(F)\rangle_{\mathbb{R}^d}\right]\\
  &=\sum_{i,j=1}^dE\left[A_{ij}\partial _i \partial _j f(F)\right]-\sum_{i=1}^d E\left[F_i \partial _i f(F)\right],
\end{align*}
where $\partial_i$ is a short notation for $\frac {\partial}{ \partial x_i}$.
    Using the duality formula in the second term, we can write 
    \begin{align}
    \Phi &= \sum_{i,j=1}^dE\left[A_{ij} \partial _i \partial _j f(F)\right]-\sum_{i=1}^dE\left[\left\langle v_i,D( \partial_i f(F))\right\rangle_{\mathcal{H}}\right]\nonumber\\
    &=\sum_{i,j=1}^dE\left[A_{ij}\partial_i\partial_j f(F)\right]-\sum_{i=1}^dE\left[\langle v_i,\sum_{j=1}^d \partial_j\partial_if(F)DF_j\rangle_{\mathcal{H}}\right]\nonumber
    \\
    &=\sum_{i,j=1}^dE\left[\left(A_{ij}-\langle v_i,DF_j\rangle_{\mathcal{H}}\right)\partial_j\partial_if(F)\right].\nonumber
    \end{align}
    Finally, this clearly implies  using the Cauchy-Schwartz inequality
    \[
|\Phi| \leq  \left(\sum_{i,j=1}^d   E(|\partial_j \partial_i f(F)|^2)    \sum_{i,j=1}^d
     E(|A_{ij}-\langle v_i, DF_j\rangle|_{\mathcal{H}}^2) \right)^{\frac 12},
\]
which completes the proof.
\end{proof}

In the next proposition we will apply  Proposition \ref{bound1} to the random vector $F^R(t)$  defined in \eqref{avg}.   Notice  that from 
\eqref{kernel} we have the representation  $F_i^R(t)=\delta(v^R_i(t))$ for $i=1,\dots, d$, where
\begin{equation} \label{v}
    v^R_{i,k}(t)(s,y)=\bold{1}_{[0,t]}(s)\frac{1}{\sqrt{R}} \sigma_{ik}(u(s,y)) \int_{-R}^R p_{t-s}(x-y) dx, \quad 1\le k \le m.
\end{equation}

\begin{proposition}\label{bound3}
Let $F^R(t)$ be as defined in \eqref{avg}, and  let $v^R(t)=(v^R_1(t),\dots, v_d^R(t))$ be  as defined in \eqref{v} for $i=1,\dots, d$.
Suppose that $N^R(t)$  is a $d$-dimensional centered Gaussian random variable
  with covariance matrix $C^R(t)$ such that $C^R_{ij}(t)=E[F^R_i(t) F^R_j(t)]$ for all $i,j =1,\dots, d$.  Then,
\begin{align}
    d_W(F^R(t),N^R(t))&\leq  \sqrt{d}\|(C^R(t))^{-1}\|_{\rm op} \| C^R(t)\|_{\rm op}^{\frac12}   \sqrt{ \sum_{i,j=1}^d {\rm Var}(\langle v^R_i(t),DF^R_j(t)\rangle_{\mathcal{H}})}.
\end{align}
\end{proposition}

\begin{proof}
By the duality formula \eqref{duality} we can write
\[
 C^R_{ij}(t)=E(F^R_i(t) F^R_j(t)) =E(\delta(v^R_i(t)) F^R_j(t)) =E[\langle v^R_i(t), DF^R_j(t)\rangle_{\mathcal{H}}].
 \]
 As a consequence, from \eqref{EQ1} we have
 \begin{align*}
    d_W(F^R(t),N^R(t))&  \leq  \sqrt{d}\|(C^R(t))^{-1}\|_{\rm op} \| C^R(t)\|_{\rm op}^{\frac12}\\
    & \qquad \times  \sqrt{ \sum_{i,j=1}^d   E\left[ \left| |E[\langle v^R_i(t), DF^R_j(t)\rangle_{\mathcal{H}}]-\langle v^R_i(t), DF^R_j(t)\rangle_{\mathcal{H}}\right| ^2
    \right]}\\
    &=  \sqrt{d}\|(C^R(t))^{-1}\|_{\rm op} \| C^R(t)\|_{\rm op}^{\frac12} \sqrt{\sum_{i,j=1}^d  {\rm Var}(\langle v^R_i(t), DF^R_j(t)\rangle_{\mathcal{H}})}.
\end{align*}
This completes the proof of the proposition.
\end{proof}

Next we compute the entries of the asymptotic covariance matrix of  the process $\{F^R(t), t\in [0,T]\}$.
\begin{proposition}\label{asympcov}
Let $G^R_i(t)=\int_{-R}^Ru_i(t,x)dx-2R$ for any $i=1,\dots, d$. Then
\[
\lim_{R\rightarrow\infty}\frac{1}{R}E\left[G_i^R(t) G^R_j(s)\right]=2\sum_{k=1}^m\int_0^{t\wedge s}\eta^{(k)}_{ij}(r)dr,
\]
where $\eta^{(k)}_{ij}(r)=E\left[ \sigma_{ik}(u(r,x))\sigma_{jk}(u(r,x))\right]$ has been defined in \eqref{nu}.
\end{proposition}
\begin{proof}
The proof is  similar to the one done in Proposition 3.1 of \cite{CLSHE}. For the sake of completion we will give some details below.
We can write
\begin{align*}
    E[u_i(t,x)u_j(s,x')]&=1+\sum_{k=1}^m\int_0^{t \wedge s}\int_{\mathbb{R}}p_{t-r}(x-y)p_{s-r}(x'-y)E[\sigma_{ik}(u(r,x))\sigma_{jk}(u(r,x))]dydr\\
    &=1+\sum_{k=1}^m\int_0^{t\wedge s}\int_{\mathbb{R}}\eta^{(k)}_{ij}(r)p_{t-r}(x-y)p_{s-r}(x'-y)dydr\\&=1+\sum_{k=1}^m\int_0^{t\wedge s}\eta^{(k)}_{ij}(r)p_{t+s-2r}(x-x')dr.
\end{align*}
Therefore,
\begin{align*}
    {\rm Cov}\left(G^R_i(t),  G^R_j(t)\right)&=\sum_{k=1}^m\int_{-R}^R\int_{-R}^R\int_0^{t\wedge s}\eta^{(k)}_{ij}(r)p_{t+s-2r}(x-x')dr\,dx\,dx'\\&=2\sum_{k=1}^m\int_0^t\eta^{(k)}_{ij}(r)\int_0^{2R}p_{t+s-2r}(z)(2R-z)dzdr.
\end{align*}
As a consequence,
\begin{align*}
    \lim_{R\rightarrow \infty}\frac{1}{R} {\rm Cov}\left(G^R_i(t), G^R_j(s)\right) &= \lim_{R\rightarrow \infty} 2\sum_{k=1}^m\int_0^{t\wedge s}\eta^{(k)}_{ij}(r)\int_0^{2R}p_{t+s-2r}(z)(2-\frac{z}{R})dzdr\\&=2\sum_{k=1}^m\int_0^{t \wedge s}\eta^{(k)}_{ij}(r)dr,
\end{align*}
which completes the proof of the proposition.
\end{proof}

\section{Proof of the main results}
In this section we will show Theorems \ref {MT} and  \ref{CLT}.
\subsection{Proof of Theorem \ref{MT}}
\begin{proof}
Let us  recall that  we used $C^R(t)$ to denote the covariance  matrix of $F^R(t)$ i.e. $C^R_{ij}(t)=E(F^R_i(t) F^R_j(t))$.
In the proof of Proposition \ref{asympcov} we obtained that  
\[
  C^R_{ij}(t)=2\sum_{k=1}^m\int_0^t\eta^{(k)}_{ij}(r)\int_0^{2R}p_{2t-2r}(z)(2-\frac{z}{R})dzdr.
\]
Let $N^R(t)$ and $N(t)$ denote two $d$-dimensional  Gaussian random vectors  with covariance matrices $C^R(t)$ and $C(t)$, where $C(t)$ has been defined in \eqref{LimitMatrix}. Applying  the triangle inequality for the Wasserstein distance and can write
  \begin{align*}
    d_W(F^R(t),N)\le d_W(F^R(t),N^R(t))+d_W(N^R(t),N(t)),
\end{align*}
The proof of Theorem \ref{MT} will be done in two steps.

\medskip
\noindent
{\it Step 1.}  We claim that for any $R\ge 1$,
\begin{equation} \label{EQ6}
d_W(F^R(t),N_R(t)) \le \frac{ c} {\sqrt{R}},
\end{equation}
where $c$ is a constant depending on $t$.  To show \eqref{EQ6} we make use of the following bound proved in Proposition \ref{bound3}: 
\begin{equation} \label{EQ5}
d_W(F^R(t),N_R(t))\le \sqrt{d}\|(C^R(t))^{-1}\|_{\rm op} \| C^R(t)\|_{\rm op}^{\frac12}   \sqrt{ \sum_{i,j=1}^d {\rm Var}(\langle v^R_i(t),DF^R_j(t)\rangle_{\mathcal{H}})}.
\end{equation}
In order to  compute ${\rm Var}(\langle v^R_i(t),DF^R_j(t)\rangle_{\mathcal{H}})$, applying  Fubini's theorem we can write
\begin{align*}
    F^R_i(t)&=\frac{1}{\sqrt{R}}\left(\int_{-R}^Ru_i(t,x)dx-2R  \right)\\&=\sum_{k=1}^m\frac{1}{\sqrt{R}}\left(\int_{-R}^R\int_0^t\int_{\mathbb{R}}p_{t-s}(x-y)\sigma_{ik}(u(s,y))W_k(ds,dy)dx \right)\\&=\sum_{k=1}^m\int_0^t\int_{\mathbb{R}}\left(\frac{1}{\sqrt{R}}\int_{-R}^R p_{t-s}(x-y)\sigma_{ik}(u(s,y))dx\right)W_k(ds,dy).
\end{align*}
We recall that for any fixed $t\geq 0,F^R_i(t)=\delta(v^R_i(t))$, where   $v^R_i(t)$ is defined in \eqref{v}.
Moreover,
\begin{align*}
    D^\ell_{s,y}F^R_i(t)=\textbf{1}_{[0,t]}(s)\frac{1}{\sqrt{R}}\int_{-R}^{R}D^\ell_{s,y}u_i(t,x)dx.
\end{align*}
From \eqref{deriv}, we can write for $\ell =1,\dots, m$,
\begin{align*}
    D^\ell_{s,y}u^i(t,x)&= p_{t-s}(x-y)\sigma_{i\ell}(u(s,y))\nonumber\\& \qquad \qquad +\sum_{j=1}^m \sum_{k=1}^d\int_s^t\int_{\mathbb{R}}p_{t-r}(x-z)\Sigma^{(k)}_{ij}(r,z)D^\ell_{s,y}u_k (r,z)W_j(dr,dz),
    \end{align*}
where $\Sigma_{ij}^{(k)}(r,z)$ is the adapted process defined there.
Therefore,
\begin{align*}
&\langle DF^R_i(t),v^R_j(t)\rangle_{\mathcal{H}} \\
& \qquad =\sum_{\ell=1}^m\frac{1}{R}\int_0^t\int_{\mathbb{R}}\left(\int_{-R}^Rp_{t-s}(x-y)dx\right)^2\sigma_{i\ell}(u(s,y)\sigma_{j\ell}(u(s,y)) dyds
\\&\qquad \qquad +\frac{1}{R }\sum_{\ell=1}^m\sum_{q=1}^m\sum_{k=1}^d\int_0^t\int_{\mathbb{R}}\int_{-R}^R\int_{-R}^Rp_{t-s}(x-y)\sigma_{j\ell}(u(s,y))
\\&\qquad\qquad  \qquad \times \left(\int_s^t\int_{\mathbb{R}}p_{t-r}(x'-z)\Sigma_{iq}^{(k)}(r,z)D^\ell_{s,y} u_k(u(r,z))W_q (dr,dz)\right)dxdx'dyds\bigg)\\
&\qquad  =: \Phi_1+ \Phi_2.
\end{align*}
Now using  arguments analogous to  those in the  proof of Theorem 1.1 in \cite{CLSHE}, we deduce the bound 
\begin{equation}\label{EQ2}
{\rm Var} \left(\langle DF^R_i(t),v^R_j(t)\rangle_{\mathcal{H}} \right)
\leq c R^{-1},
\end{equation} 
where  $c$ is a constant depending on $t$.  More precisely, the proof of the estimate \eqref{EQ2} is based on the following ingredients. The variance of $\Phi_1$ is estimated using Poincar\'e's inequality  ${\rm Var} (\Phi_1)   \le  E( \| D\Phi_1\|_{\mathcal{H}}^2)$ and for the variance of $\Phi_2$ we use the isometry of the stochastic integral. Finally, we make use of the estimate $\|D_{s,y}u(r,z) \| _p \le Cp_{r-s}(z-y)$, for any  $0<s<r\le t$ and $y,z\in \mathbb{R}$, where the constant $C$ depends on $t$ and  $p$ (see \cite[Lemma A.1]{CLSHE}).

On the other hand, we claim that for a fixed $t>0$, 
\begin{equation}
\sup_{R>0} \| C^R(t) \|_{\rm op} <\infty \label{EQ3}
\end{equation}
and
\begin{equation}
\sup_{R\ge 1} \| (C^R(t))^{-1} \|_{\rm op} <\infty \label{EQ4}.
\end{equation}
Then the estimate \eqref{EQ6} will be a consequence of the bounds \eqref{EQ2}, \eqref{EQ3} and \eqref{EQ4}.

It is easy to show the estimate \eqref{EQ3}  and we omit the details. 
In order to show the claim \eqref{EQ4} it suffices to get a lower bound, uniformly in $R\ge 1$, for the determinant of the matrix $C^R(t)$. We have
\[
\det C^R(t) \ge \left( \inf _{|\xi|=1} \xi^T C^R(t) \xi \right)^d.
\]
We can obtain a lower bound for the quadratic form  $\xi^T C^R(t) \xi $ as follows
\begin{align*}
\xi^T C^R(t) \xi& =2 \sum_{i,j=1}^d \sum_{k=1}^m \int_0^t \eta_{ij}^{(k)}(r) \xi_i\xi_j \int_0^{2R} p_{2t-2r}(z) (2- \frac zR) dzdr\\
&= 2 \sum_{k=1}^m \int_0^t E \left( \left| \sum_{i=1}^d \xi_i \sigma_{ik}(u(r,0)) \right|^2\right)  \int_0^{2R} p_{2t-2r}(z) (2- \frac zR) dzdr.
\end{align*}
We have
\[
\int_0^{2R} p_{2t-2r}(z) (2- \frac zR) dz \ge \int_0^{R} p_{2t-2r}(z)   dz \ge \int_0^{1} p_{2t-2r}(z)dz.
\]  
Set
\[
\varphi_\xi(r)=2\sum_{k=1}^m E \left( \left| \sum_{i=1}^d \xi_i \sigma_{ik}(u(r,0)) \right|^2\right).
\]
Then, we can write
\[
\xi^T C^R(t) \xi \ge \int_0^t  \varphi_\xi(r)\int_0^{1} p_{2t-2r}(z)dz dr \ge \int_0^1 p_t(z) dz \int_0^{t/2} \varphi_\xi(r) dr.
\]
Our assumption {\bf (H1)} implies that for any unit vector $\xi$, 
\[
\varphi_\xi(0) =2\sum_{k=1}^m  \left( \left| \sum_{i=1}^d \xi_i \sigma_{ik}(\overline{1}) \right|^2\right) >0.
\]
This implies that  $\xi^T C^R(t) \xi >0$ and, by continuity,  $\inf _{|\xi|=1} \xi^T C^R(t) \xi>0$, which completes the proof of claim \eqref{EQ4}.

\medskip
\noindent
{\it Step 2.} We claim that for $R\ge 1$, 
 \begin{equation} \label{claim}
 d_{W}(N^R(t),N(t)) \le  c R^{-1},
 \end{equation}
 for some constant $c>0$ depending on $t$.
 The proof this estimate we make use of the following bound on the Wasserstein distance between two multidimensional Gaussian vectors
 (see, for instance, Exercise 4.5.3 in  \cite {NAMC}:
 \[
 d_{W}(N^R(t),N(t)) \le Q(C^R(t), C(t)) \| C^R(t)- C(t) \|_{\rm HS},
 \]
 where
 \[
  Q(C_R(t), C(t))= \sqrt{d} \min\{ \|(C^R(t))^{-1} \|_{\rm op} \|C^R(t)\|^{1/2} _{\rm op},   \|C(t)^{-1} \|_{\rm op} \|C(t)\|^{1/2} _{\rm op} \}.
  \]
 The estimates \eqref{EQ3} and \eqref{EQ4} imply that  $  Q(C^R(t), C(t))$ is uniformly bounded by a constant depending on $t$ for any $R\ge 1$.
 Moreover,
 \begin{align*}
  \| C^R(t)- C(t) \|_{\rm HS} &= \left[ \sum_{i,j=1}^d | C^R_{ij}(t) - C_{ij}(t)|^2 \right] ^{\frac 12} \\
  & =  \left[ \sum_{i,j=1}^d  \left| 2 \sum_{k=1}^m \int_0^t \eta^{(k)}_{ij} (r) \left( 1- \int_0^{2R} p_{2t-2r}(z) (2- \frac zR)dz \right) dr \right|^2 \right]^{\frac 12} \\
  & \le c \int_0^t \left | 1- \int_0^{2R} p_{2t-2r}(z) (2- \frac zR)dz \right| dr \\
&\le c \int_0^t \left | 1- 2\int_0^{2R} p_{2t-2r}(z) dz \right| dr +\frac cR   \int_0^t  \int_0^{2R} p_{2t-2r}(z)  |z|dz dr.
  \end{align*}

If $Z_{r,t}$ denotes a random variable with the distribution $  N(0,2t-2r)$, we can write
\[
  \| C^R(t)- C(t) \|_{\rm HS}  \le c\int_0^t P(|Z_{r,t}|>2R)dr  +  \frac cR \int_0^t E(| Z_{r,t}|) dr  \le \frac cR.
\]
This completes the proof of the estimate \eqref{claim}.
\end{proof}

\subsection{Proof of Theorem \ref{CLT}}

\begin{proof}
The proof is similar to that of Theorem 1.2 in  \cite{CLSHE} and we  skip the details.  It suffices to show the weak convergence of finite-dimensional distributions and the tightness property. Tightness follows from the next lemma, whose proof is analogous to that of Proposition 4.1 in   \cite{CLSHE}.

\begin{lemma}\label{tight}
Let $u(t,x)=(u_1(t,x),\dots,u_d(t,x))$ be the solution to equation \eqref{equation}. Then for  each $i, 1\le i\le d$ and any $0\le s<t\le T$ and any $p\ge 1$ there exists a constant $K=K(p,T)$ such that 
\begin{align*}
   E\left( \left |\int_{-R}^R u_i(t,x)dx-\int_{-R}^R u_i(s,x)dx \right |^p\right)&\le KR^{\frac{p}{2}}(t-s)^{\frac{p}{2}}.
\end{align*}
\end{lemma}
 
 In order to show the convergence in law of the finite-dimensional distributions,   we fix points
 $0\le t_1\le   \cdots \le t_M \le T$. Recall that for each $i_0,1\le i_0\le d$
\[
F_{i_0}^R(t_i)=\frac{1}{\sqrt{R}}\left(\int_{-R}^{R}u_{i_0}(t_i,x)dx-2R\right)
\] 
and set, for $1\le i,j \le M$ and $1\le p,q\le d$,
 \[
 A^{ij}_{pq}:=2 \sum_{k=1}^m \int_0^{t_i\wedge t_j}\eta^{(k)}_{pq}(r)dr, \quad 1\le i,j \le M,
 \]
where  $\eta^{(k)}_{pq}(r)$ is defined in \eqref{nu}.
With arguments similar to those the proof of Theorem 1.2 in \cite{CLSHE}, we can show that for all $1\le i,j \le M$ and $1\le p,q\le d$,
 \[
 \lim_{R\rightarrow \infty}E\left[ \left(A^{ij}_{pq}-\langle DF^R_p(t_i),v^R_{q}(t_j)\rangle_{\mathcal{H}}\right)^2\right]=0.
 \]   
We can then complete the proof  by the methodology used in the proof of  Theorem 1.2 in  \cite{CLSHE} 
\end{proof}

\end{document}